\documentclass[11pt]{article}
\usepackage{multicol}
\usepackage{amsmath,amssymb,amsthm,array,booktabs,mathtools,tabularray}
\usepackage{mathpazo}
\usepackage[vmargin=2.7cm, hmargin=2.8cm]{geometry}
\usepackage[justification=centering]{caption}
\newcolumntype{C}[1]{>{\centering\arraybackslash}p{#1}}
\theoremstyle{definition}
\newtheorem{thm}{Theorem}
\newtheorem{defn}{Definition}
\newtheorem{corr}{Corollary}
\newtheorem{lemma}{Lemma}
\newtheorem{prop}{Proposition}
\newtheorem{conj}{Conjecture}

\newtheorem{remark}{Remark}
\newcommand{\stern}[2]{S_{#1}(#2)}
\renewcommand{\Re}{\operatorname{Re}}
\renewcommand{\Im}{\operatorname{Im}}
\newcommand{\lm}{\lambda}
\newcommand{\ball}[2]{B_{#2}(#1)}
\newcommand{\cball}[2]{B_{#2}[#1]}
\usepackage{chngcntr}
\counterwithin*{equation}{section}

\newcommand{\C}{\mathbb C}
\newcommand{\Q}{\mathbb Q}
\newcommand{\Z}{\mathbb Z}
\newcommand{\N}{\mathbb N}
\newcommand{\R}{\mathbb R}
\newcommand{\calA}{\mathcal A}
\newcommand{\calB}{\mathcal B}

\newcommand{\calE}{\mathcal E}
\newcommand{\calO}{\mathcal O}
\newcommand{\calS}{\mathcal S}

\newcommand{\abs}[1]{\left|#1\right|}
\newcommand{\eps}{\varepsilon}

\title{Zeros of Stern polynomials in the complex plane}
\author{David Altizio}
\begin{document}

\maketitle

\renewcommand{\baselinestretch}{1.05}
\begin{abstract}
	The classical Stern sequence of positive integers was extended to a polynomial sequence $\stern n\lm$ by Klav\v{z}ar et. al. by defining $\stern 0 \lm = 0$, $\stern 1 \lm = 1$, and 
	\[
	\stern{2n}\lm = \lm \stern n\lm,\quad \stern{2n+1}\lm = \stern n\lm + \stern{n+1}\lm.
	\]
	Dilcher et. al. conjectured that all roots of $\stern n\lm$ lie in the half-plane $\{\Re w < 1\}$.  We make partial progress on this conjecture by proving that $\{|w-2| \leq 1\}\subseteq\C$ does not contain any roots of $\stern n\lm$.  Our proof uses the Parabola Theorem for convergence of complex continued fractions.  As a corollary, we establish a conjecture of Ulas and Ulas by showing that $\stern p\lm$ is irreducible in $\Z[\lm]$ whenever $p$ is a positive prime.
\end{abstract}

\setlength{\parskip}{0pt}
\tableofcontents

\setlength{\parskip}{8pt}
\renewcommand{\baselinestretch}{1.1}

\section{Introduction}

The Stern sequence $(s_n)_{n\geq 0}$ of positive integers, named after Moritz Abraham Stern, is given by $s_0 = 0$, $s_1 = 1$, and 
\begin{equation}\label{eqn:stern-seq}
	s_{2n} = s_n,\quad s_{2n+1} = s_n + s_{n+1}.
\end{equation}
There is a vast body of literature (e.g. \cite{Coons2014,Northshield2010,Reznick2008,Stanley2019} and the references therein) about the Stern sequence.  Perhaps the most striking property of the Stern sequence is that every positive rational number appears exactly once in the sequence $(s_{n+1}/s_n)_{n\geq 1}$, giving an explicit enumeration of $\mathbb Q^+$.  Stern proved this fact 15 years before Cantor introduced the notion of a countable set!

\par In \cite{Klavzar2007}, Klav\v{z}ar, Milutinovi\'{c}, and Petr define a polynomial analogue $\stern n\lm$ of the Stern sequence via $\stern 0\lm = 0$, $\stern 1\lm = 1$, and 
\begin{equation}\label{eqn:stern-def}
	\stern{2n}\lm = \lm \stern n\lm,\quad \stern{2n+1}\lm = \stern n\lm + \stern{n+1}\lm.
\end{equation}
The first $16$ terms of the sequence $\stern n\lm$ are given in Table \ref{table:stern-16}.  This table suggests many patterns, several of which are recorded in Proposition \ref{prop:basic-patterns}.

\begin{table}[ht]
	\[
\def\arraystretch{1.2}
\begin{array}{c|l||c|l}
	n & \stern n\lambda & n & \stern n\lambda  \\ \hline
	1 & 1 & 9 & 1 + 2\lambda + \lambda^2 \\
	2 & \lambda  & 10 & \lambda + 2 \lambda ^2  \\
	3 & 1+\lambda & 11 & 1 + 3\lambda + \lambda ^2 \\
	4 & \lambda ^2 & 12 & \lambda ^2 + \lambda ^3 \\
	5 & 1+ 2 \lambda  & 13 & 1 + 2\lambda + 2 \lambda ^2 \\
	6 & \lambda+\lambda^2  & 14 & \lambda+\lambda ^2+\lambda^3  \\
	7 & 1 + \lambda+\lambda^2 & 15 & 1 + \lambda +\lambda ^2+\lambda^3 \\
	8 & \lambda ^3 & 16 & \lambda ^4 \\
\end{array}
\]
	\caption{\small The first $16$ terms of the sequence $\stern n\lm$.}
	\label{table:stern-16}
\end{table}

While there are several papers which analyze the sequence $\stern n\lm$ (see \cite{Dilcher2017,Dilcher2018,Gawron2014,Ulas2011} and the references therein), not much is known about the zeros of Stern polynomials.  Because $\stern{2n}\lm = \lm \stern n \lm$ for all $n\geq 1$, it suffices to examine the roots of $\stern n\lm$ only when $n$ is odd.  Define
\[
\calS\coloneqq \{z\in\C: \stern n z = 0\text{ for some odd }n\geq 1\}.
\]
In \cite[Theorem 2.1]{Gawron2014}, Gawron proves that $0$, ${-}1$, ${-}\tfrac12$, and $-\tfrac13$ are the only rational zeros of any Stern polynomial by showing, more generally, that any real number in the closed interval $[-\tfrac14,\tfrac14]$ cannot be in $\calS$.  His methods extend easily to show that any complex number $z$ with $|z| \leq \tfrac14$ cannot be in $\calS$.  For completeness, we record the proof here.

\begin{thm}\label{thm:gawron-generalization}
	If $z\in \calS$, then $|z| > \tfrac 14$.
\end{thm}

\begin{proof}
	Let $z$ be any complex number with $|z| \leq \tfrac 14$.  Let $b_n\coloneqq |\stern n z|$ for each positive integer $n$.  We show more strongly that
	\begin{equation}\label{eqn:stronger-statement}
		b_{2n+1} > \tfrac12\max\{b_n,b_{n+1}\} > 0
	\end{equation}
	for all $n\geq 1$.
	
	\par The proof of \eqref{eqn:stronger-statement} proceeds by induction on $n$.  The base case, $n = 1$, follows because
	\[
	b_3 = |z + 1| \geq \tfrac 34 > \tfrac 12 
	= \tfrac 12\max\{1, |z|\} = \tfrac 12\max\{b_1,b_2\}.
	\]
	
	\par There are two cases to consider.  First assume $n=2k$ is even.  Then
	\begin{align*}
		b_{4k+1} &= |\stern{4k+1}z| = |\stern{2k+1} z + \stern{2k} z|\\
		&= |z \stern kz + \stern{2k+1} z| \geq b_{2k+1} - \tfrac14 b_k\\
		&\geq b_{2k+1} - \tfrac12 b_{2k+1} = \tfrac12 b_{2k+1}.
	\end{align*}
	Moreover, $b_{2k+1} \geq \frac12 b_k$, and thus $\max\{b_{2k},b_{2k+1}\} = b_{2k+1} > 0$.  In this case, our inequality is proved.
	
	\par Now assume $n = 2k + 1$ is odd.  Then
	\begin{align*}
		b_{4k+3} &= |\stern{4k+3}z| = |\stern{2k+1} z + \stern{2k+2} z|\\
		&= |\stern{2k+1} z + z\stern{k+1} z| \geq b_{2k+1} - \frac 14 b_{k+1}\\
		&\geq b_{2k+1} - \tfrac12 b_{2k+1} = \tfrac12 b_{2k+1}.
	\end{align*}
	Moreover, $b_{2k+1} \geq \frac12 b_{k+1}$, and so in this case we also have $\max\{b_{2k},b_{2k+1}\} = b_{2k+1} > 0$.  We have exhausted both cases, completing the proof of Theorem \ref{thm:gawron-generalization}.
\end{proof}

In \cite{Dilcher2017}, Dilcher et. al. focus more specifically on the complex roots of $\stern n\lm$.  Their paper makes the following conjecture.

\begin{conj}\label{conj:location-roots}
	All elements of $\calS$ lie in the half-plane $\{\Re w < 1\}$. 
\end{conj}

By generalizing the Enestrom-Kakeya theorem, they prove Conjecture \ref{conj:location-roots} for several classes of positive integers $n$ taking the form $2^n\pm k$, where $k$ is fixed and $2^n \geq k$.  These are the only two papers the author could find which discuss the complex zeros of $\stern n\lm$.

\par Figure \ref{fig:stern-roots} shows a snapshot of $\calS$.  One striking feature of this figure is the contrasting behavior of these roots within the half-planes $\{\Re w\geq 0\}$ and $\{\Re w < 0\}$.  These differences present difficulties in fully characterizing the geometry of $\calS$.

\begin{figure}[ht]
	\centering
	\captionsetup{width=0.75\linewidth}
	\includegraphics[scale=0.75]{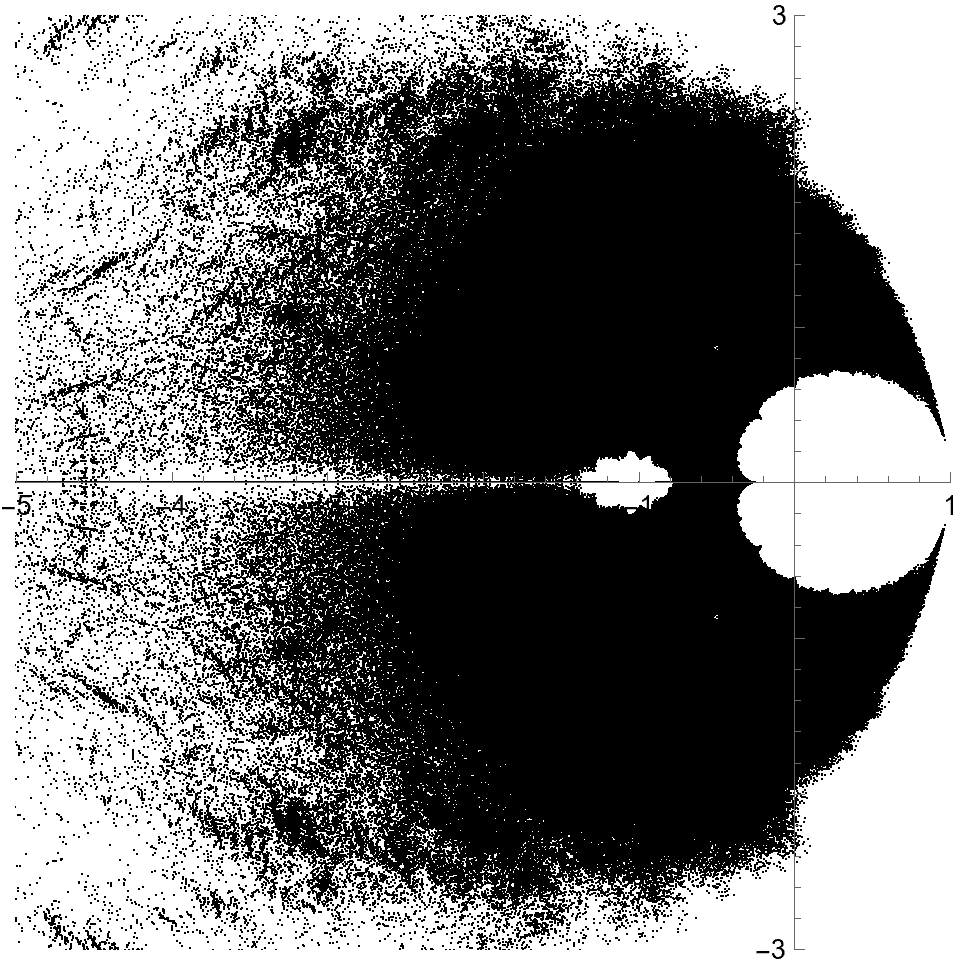}
	\caption{\small The zeros of $\stern n\lm$ in the range $\{a+bi:-4\leq a\leq 1, |b|\leq 3\}$, where $1\leq n<2^{21}$ is odd.}
	\label{fig:stern-roots}
\end{figure}

In this paper, we partially resolve Conjecture \ref{conj:location-roots} by establishing the following result.

\begin{thm}\label{thm:roots-disk}
	All roots of $\stern n\lm$ lie outside the disk $\{\abs{z-2}\leq 1\}\subseteq\C$.
\end{thm}

To prove Theorem \ref{thm:roots-disk}, we use a continued fraction representation for a ratio of Stern polynomials (Theorem \ref{thm:stern-cont-frac}) independently discovered by Reznick \cite{Reznick2008} and Schinzel \cite{Schinzel2014}.  This allows us to use the Parabola Theorem (Theorem \ref{thm:parabola}) to show that, for certain values of $z\in\mathbb C$, the denominators of these continued fractions can never be zero.  Along the way, we establish inequalities in $\C$ relating to the sums $1 + z + \cdots + z^{n-1}$ which may be of independent interest.  The most notable of these inequalities is Theorems \ref{thm:MO-ineq}, which proves a lower bound for this geometric series whenever $\Re z \geq 1$.

As a corollary, we obtain the following surprising fact, resolving a conjecture of Ulas and Ulas (\cite{Ulas2011}).

\begin{corr}\label{thm:prime-poly-irred}
	For each prime number $p$, the Stern polynomial $\stern p \lm$ is irreducible in $\Q$.
\end{corr}
Ulas and Ulas had verified this conjecture computationally for the first million primes $p$.  Additionally, Schinzel in \cite{Schinzel2011} proved Corollary \ref{thm:prime-poly-irred} for all primes $p < 2017$ by using finite differences to bound the leading coefficient of any proper divisor of $\stern n\lm$.  However, these previous attempts to prove the conjecture were algebraic in nature, whereas our proof depends on the analytic properties of $\calS$.


\paragraph{Acknowledgments.} This paper is adapted from the author's PhD dissertation \cite{Altizio2025}.  The author thanks his advisor, Dr. Bruce Reznick, for helpful correspondence.

\section{Stern Polynomial Preliminaries}\label{sec:prelims}

\paragraph{Notation.}
Write $\C$ as usual for the field of complex numbers, and let $\hat\C \coloneqq \C\cup\{\infty\}$ be the extended complex plane. For all positive integers $n$, we introduce the (non-standard) notation
\[
(\lm)_n\coloneqq 1 + \lm + \cdots + \lm^{n-1} = \frac{\lm^n - 1}{\lm - 1}\in \Z[\lm].
\]
For continuity purposes, we set $(1)_n = n$.  For any set $S\subseteq \C$, define the auxiliary sets $S^+$ and $S^-$ via
\[
S^+\coloneqq S\cap \{\Im w \geq 0\}\quad\text{and}\quad S^-\coloneqq S \cap\{\Im w \leq 0\}.
\]
Furthermore, for $U\subseteq\hat\C$, set $U^{-1} \coloneqq \{w^{-1}:w\in U\}\subseteq\hat\C$.  Finally, for all $z\in\C$ and $r\geq 0$, let $\ball z r$ be the open Euclidean ball centered at $z$ with radius $r$, and let $\cball z r$ denote the closed Euclidean ball centered at $z$ with radius $r$. 

We begin by recording some basic properties of the Stern polynomials.  All of these can be proven by induction.
\begin{prop}\label{prop:basic-patterns}
	Let $n$ be a nonnegative integer.
	
	\vspace{-3pt}
	
	\begin{enumerate}
		\item $\stern{n}{0} = 0$ when $n$ is even, and $\stern n 0 = 1$ when $n$ is odd.
		\item $\stern{n}{1} = s_n$, where $s_n$ is the Stern sequence \eqref{eqn:stern-seq}.
		\item $\stern{n}{2} = n$.
		\item When $n = 2^r$ for a nonnegative integer $r$, $\stern n \lm = \lm^r$.
		\item When $n = 2^r - 1$ for a nonnegative integer $r$, $\stern n\lm = (\lm)_r$.
	\end{enumerate}
\end{prop}

One of the most fundamental properties of the Stern sequence is the recursion
\[
s_{2^km + r} = s_{2^k - r}s_m + s_rs_{m+1},
\]
valid for $m\geq 0$ and $2^k > r \geq 0$.  The Stern polynomials satisfy an analogous recursion.

\begin{prop}[{\cite[Lemma 1]{Schinzel2014}}] \label{prop:important-recursion}
	For all integers $m\geq 0$ and $2^k> r\geq 0$,
	\[
	\stern{2^km+r}\lm = \stern{2^k-r}\lm\stern m\lm + \stern r\lm \stern{m+1}\lm.
	\]
\end{prop}

\begin{corr}\label{cor:special-case-stern}
	For positive integers $a$ and $b$, 
	\[
	\stern{2^{a+b} - 2^b + 1}\lm = (\lm)_a(\lm)_b + \lm^a.
	\]
\end{corr}

\begin{proof}
	Write $2^{a+b} - 2^b + 1 = 2^b(2^a - 1) + 1$.
	Apply Proposition \ref{prop:important-recursion} with $k = b$, $m = 2^a - 1$, and $r = 1$ to obtain
	\begin{align*}
		\stern{2^{a+b} - 2^b + 1}\lm &= \stern{2^b-1}\lm\stern{2^a-1}\lm+\stern 1\lm \stern{2^a}\lm \\
		&= (\lm)_b(\lm)_a + \lm^a.\qedhere
	\end{align*}
\end{proof}

Proposition \ref{prop:important-recursion} may be restated in a form that highlights its underlying structure.  Note that any odd integer can be expressed in the form
\begin{equation}\label{eqn:bracket-def}
	[[a_1,\cdots, a_t]] := 2^{a_1 + \cdots + a_t} - 2^{a_2 + \cdots + a_t} + \cdots + (-1)^{t-1}2^{a_t} + (-1)^t,
\end{equation}
where $a_1,\ldots, a_t$ are positive integers.  For example, $2^{a+b} - 2^b + 1 = [[a,b]]$, so Corollary \ref{cor:special-case-stern} takes the form
\begin{equation}\label{eqn:special-case-exmp}
	\stern{[[a,b]]}\lm = (\lm)_a(\lm)_b + \lm^a.
\end{equation}
The notation \eqref{eqn:bracket-def} satisfies the recursive properties
\begin{align}
	[[a_1,\ldots, a_{t-1},a_t]] &= 2^{a_1 + \cdots + a_t} - [[a_2,\ldots, a_t]] \label{eqn:recur-head}\\
	&= 2^{a_t} [[a_1,\ldots, a_{t-1}]]+ (-1)^t.\label{eqn:recur-tail}
\end{align}

\begin{thm}[Stern Polynomial Recursion]\label{thm:stern-poly-recursion}
	For all sequences of positive integers $a_1$, $\ldots$, $a_t$,
	\begin{equation}\label{eqn:stern-recursion}
		\stern{[[a_1,\ldots, a_t]]}\lm =
		(\lm)_{a_1}\stern{[[a_2,\ldots, a_t]]}\lm + \lm^{a_1}\stern{[[a_3,\ldots, a_t]]}\lm.
	\end{equation}
\end{thm}

\begin{proof}
	For simplicity let $r \coloneqq a_1 + \cdots + a_t$.  Define
	\begin{align*}
		n &= [[a_1,\ldots, a_t]],\\
		n' &= [[a_2,\ldots, a_t]],\quad\text{and}\\
		n'' &= [[a_3,\ldots, a_t]].
	\end{align*}
	By repeated application of \eqref{eqn:recur-head},
	\[
	n = 2^r - n' = 2^r - (2^{r-a_1} - n'') = 2^{r - a_1}(2^{a_1} - 1) + n''.
	\]
	Applying Proposition \ref{prop:important-recursion} with $k = r - a_1$, $m = 2^{a_1} - 1$, and $p = n''$ yields
	\begin{align*}
		\stern{n}\lm &= \stern{2^{r-a_1} - n''}\lm \stern{2^{a_1} -1}\lm + \stern{n''}\lm \stern{2^a}\lm \\
		&= \stern{n'}\lm(\lm)_{a_1} + \stern{n''}\lm \lm^{a_1}.\qedhere
	\end{align*}
\end{proof}

Equation \eqref{eqn:stern-recursion} bears similarities identities satisfied by continuants of continued fractions (see \cite{Khinchin1997}).  This suggests the Stern polynomials have a continued fraction representation in terms of the sequence $a_i$, which we now state and prove.

\begin{corr}[{\cite[Theorem 1]{Schinzel2014}}]\label{thm:stern-cont-frac}
	Let $a_1,\ldots, a_t$ be positive integers, where $t\geq 2$.  Then 
	\begin{equation}\label{eqn:stern-cont-frac}
		\frac{\stern{[[a_1,\ldots, a_t]]}{\lm}}{\stern{[[a_2,\ldots, a_t]]}{\lm}}
		=
		(\lm)_{a_1} + \cfrac{\lm^{a_1}}{(\lm)_{a_2} + \cfrac{\lm^{a_2}}{\cdots + \cfrac{\cdots}{(\lm)_{a_{t-1}} + \cfrac{\lm^{a_{t-1}}}{(\lm)_{a_t}}}}}\,.
	\end{equation}
\end{corr}

\begin{proof}
	Proceed by induction on $t$.  The base case $t = 2$ follows from \eqref{eqn:special-case-exmp} and the calculation
	\[
	\frac{\stern{[[a,b]]}\lm}{\stern{[[b]]}\lm} = \frac{(\lm)_a(\lm)_b+\lm^a}{(\lm)_b} = (\lm)_a + \frac{\lm^a}{(\lm)_b}.
	\] 
	For the inductive step, apply Theorem \ref{thm:stern-poly-recursion}.
\end{proof}
Corollary \ref{thm:stern-cont-frac} will be the workhorse for our analysis of the roots of $\stern n\lm$.

\section{Continued Fractions and the Parabola Theorem}

Recall that, for a complex number $b_0$ and any finite sequences of complex numbers $a_1,\ldots, a_t$ and $b_1,\ldots, b_t$, the (generalized) continued fraction $b_0 + \boldsymbol{K}(a_n|b_n)$ is the quotient
\[
b_0 + \boldsymbol{K}(a_n|b_n) \coloneqq b_0 + \cfrac{a_1}{b_1 + \cfrac{a_2}{\cdots + \cfrac{a_t}{b_t}}}.
\]
While continued fractions are most often studied in number theory, there has been significant study toward the \textit{convergence problem} for generalized continued fractions -- that is, determining sufficient conditions on $(a_n)_{n=1}^\infty$ and $(b_n)_{n=1}^\infty$ for which the infinite continued fraction $b_0 + \boldsymbol{K}(a_n|b_n)$ converges in $\C$.  In this regard, one may also consider a continued fraction as a sequence of fractional linear transformations
\begin{equation}\label{eqn:geom-cont}
	S_n = s_0\circ s_1\circ\cdots\circ s_n,
\end{equation}
where 
\[
s_0(w) \coloneqq b_0 + w\quad\text{and}\quad s_n(w) \coloneqq \frac{a_n}{b_n + w}.
\]
This connection between continued fractions and fractional linear transformations was first discovered by Weyl \cite{Weyl1910}. While we will not need convergence of continued fractions, we will borrow some of the techniques and theorems from this field.  For a more detailed investigation, consult \cite[Chapter 3]{Lorentzen2008} and the references therein.

\par One way to analyze a generalized continued fraction is to find an \textbf{element set} and a \textbf{value set} for this fraction.  We present the definitions for these sets in the special case where $b_j = 1$ for all $j$; later, we will see that this presents no loss of generality.

\begin{defn}\label{defn:value-set}
	A set $V\subset \hat\C$ is a \textbf{value set} for $\boldsymbol{K}(a_n|1)$ if
	\[
	\frac{a_n}{1 + V}\subseteq V\text{ for }n=1,2,3,\ldots.
	\] 
\end{defn}

\begin{defn}\label{defn:element-set}
	For a given value set $V$, the set $E\subset\hat \C$ given by 
	\[
	E\coloneqq \left\{a\in \C\,{:}\, \frac{a}{1+V}\subseteq V\right\}
	\]
	is called the \textbf{element set} corresponding to the value set $V$.
\end{defn}

Observe that if $a_n\in E$ for all $n$, and $0\in V$, then $\boldsymbol{K}(a_n|1) = S_n(0)\in V$ for all $n\geq 0$.  This allows us to bound $\boldsymbol{K}(a_n|1)$ in $\hat\C$ using geometric arguments that are more precise than those found through just the Triangle Inequality.

\par One may hope that, given an element set $E$ for a continued fraction, we may find a value set $V$ corresponding to $E$.  However,
in most cases, this is a nearly impossible task. It is significantly easier
to proceed in the opposite direction: fix a value set $V$ and find an element set $E$ corresponding to $V$.  There are many theorems in the literature which take this form.  We state the most relevant result for us below.

\begin{thm}[Parabola Theorem, {\cite[Theorem 3.43]{Lorentzen2008}}]\label{thm:parabola}
	For fixed $|\alpha| < \tfrac\pi 2$, let
	\[
	V_\alpha\coloneqq -\tfrac12 + e^{i\alpha}\overline{\mathbb H} = \{w\in\C: \Re(we^{- i\alpha})\geq -\tfrac12\cos\alpha\}
	\]
	and
	\[
	E_\alpha \coloneqq \{a\in\C: |a| - \Re(ae^{-2i\alpha})\leq \tfrac12\cos^2\alpha\}.
	\]
	Then $E_\alpha$ is the element set for continued fractions $\boldsymbol{K}(a_n | 1)$ corresponding to the  value set $V_\alpha$.
\end{thm}
The region $V_\alpha$ is a half-plane whose boundary is a line intersecting the real axis at $z = -\tfrac12$.  In particular, because $|\alpha| < \tfrac\pi 2$, $V_\alpha$ contains the half-line $[-\tfrac12,\infty)$.  The boundary of $E_\alpha$ is a parabola with focus at the origin and vertex at $-\tfrac14 e^{2i\alpha}\cos^2\alpha$.  This parabola intersects the real axis at $z = -\tfrac14$.

\begin{remark}
	Some sources include $\infty$ in the half-plane $V_\alpha$.  However, because $-1\notin V_\alpha$, we know \textit{a priori} that $S_n(w) = \tfrac{a_n}{1+S_{n-1}(w)}$ never equals $\infty\in\hat\C$.  We choose to ignore $\infty$ to make the presentation cleaner.
\end{remark}

We now tie the continued fraction theory back to Stern polynomials.
Recall that Corollary \ref{thm:stern-cont-frac} expresses the ratio
of two Stern polynomials as a continued fraction with elements in $\C$. 
This form on its own is difficult to work with.  We instead note (e.g. \cite[Corollary 2.15]{Lorentzen2008}) that any continued fraction $\boldsymbol{K}(a_n|b_n)$ with $b_j\neq 0$ is equivalent to the continued fraction $\boldsymbol{K}(c_n|1)$, where $c_j = \tfrac{a_j}{b_jb_{j+1}}$.  As an explicit example, 
\begin{equation}\label{eqn:cfrac-example}
	\frac{\stern{[[2,3,5]]}\lm}{\stern{[[3,5]]}\lm} = (\lm)_2 + \cfrac{\lm^2}{(\lm)_3 + \cfrac{\lm^3}{(\lm)_5}}
	= (\lm)_2 \left[1 + \cfrac{\dfrac{\lm^2}{(\lm)_2(\lm)_3}}{1 + \cfrac{\lm^3}{(\lm)_3(\lm)_5}}\right].
\end{equation}
\begin{remark}
	While continued fractions of the form $\boldsymbol{K}(1|d_n)$ are more common, the coefficients $d_n$ are significantly messier than the coefficients $c_n$ (see \cite[Corollary 2.15]{Lorentzen2008}).  We opt to use the less-common form $\boldsymbol{K}(c_n|1)$ to simplify our analysis.
\end{remark}
Equation \eqref{eqn:cfrac-example} suggests that the ratios $\tfrac{z^a}{(z)_a(z)_b}$, where $a$ and $b$ are any positive integers, may play some importance.
For this reason, we define $z_{a,b}\coloneqq \tfrac{z^a}{(z)_a(z)_b}$ and  
\[
\calA_z\coloneqq \left\{z_{a,b}:(a,b)\in\N^2\right\}.
\]
Suppose $\calA_z\subseteq E_\alpha$ for some angle $\alpha$ with $|\alpha| < \tfrac{\pi}2$.  
Then $V_\alpha$ is a value set for the continued fraction \eqref{eqn:stern-cont-frac} at $\lm = z$.  In particular, $\infty$ is not a possible value for this fraction, so $\stern n z \neq 0$ for any $z$.

\par Throughout the rest of the paper, we use $\alpha = \tfrac{\pi}{12}$ in our applications of Theorem \ref{thm:parabola}.  In this case, Theorem \ref{thm:parabola} reduces to the following result.

\begin{thm}\label{thm:criterion-nonzero}
	Let $E$ denote the parabola
	\begin{align*}
		E\coloneqq E_{\pi/12} &= \left\{a\in\C: |a| - \Re (ae^{-\pi i/6})\leq \frac{2+\sqrt 3}8\right\}\\
		&= \left\{x+yi\in\C: \sqrt{x^2+y^2} \leq \frac{\sqrt 3}2x + \frac 12 y + \frac{2+\sqrt 3}8\right\}.
	\end{align*}
	Suppose $z\in\C$ is a complex number such that 
	$z_{a,b}\in E$ for all $(a,b)\in\N^2$.  
	Then $z$ is not a root of any Stern polynomial.
\end{thm}

The remainder of this paper is dedicated to showing $z_{a,b}\in E$ whenever $z$ lies in the set
\[
\calB\coloneqq \cball{2}{1} = \{\abs{w-2} \leq 1\},
\]
thus proving Theorem \ref{thm:criterion-nonzero}.  See Figure \ref{fig:Az-sample} for an example.

\begin{figure}[ht]
	\centering
	\includegraphics[scale=0.8]{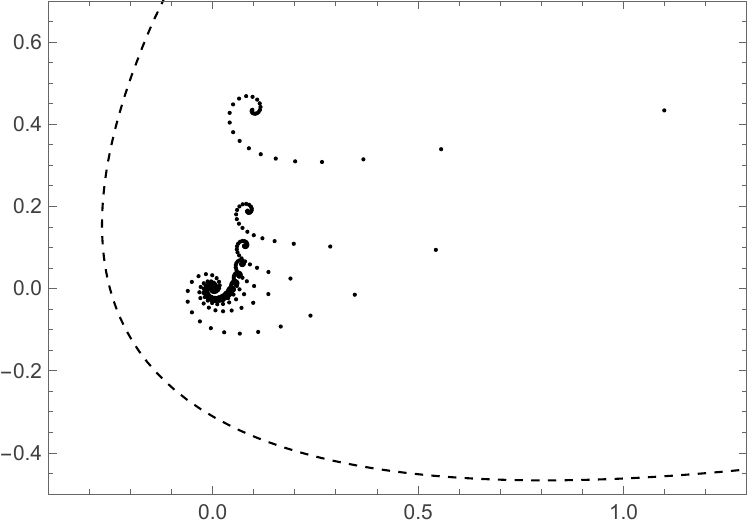}
	\caption{\small The set $\calA_z$, where $z = (2-\cos\tfrac\pi 7) + i\sin\tfrac \pi 7\in\partial\calB$, enclosed in  $E_{\pi/12}$ (dashed).}
	\label{fig:Az-sample}
\end{figure}

Let us see how Theorem \ref{thm:criterion-nonzero} proves Theorem \ref{thm:prime-poly-irred}.

\begin{proof}[Proof (Theorem \ref{thm:prime-poly-irred}).]
	Suppose $P$ is a factor of some Stern polynomial with $\deg P = d$, and factor it over $\C$:
	\[
	P(\lm) = 1 + c_1\lm + \cdots + c_d\lm^d = c_d\prod_{k=1}^d(\lm-\mu_k).
	\]
	Then $P(2) = c_d\prod_{k=1}^d (2-\mu_k)$.  Because $c_d \geq 1$ and $|2-\mu_k| > 1$ by Theorem \ref{thm:criterion-nonzero}, it follows that $P(2) > 1$.  Thus, if $\stern p\lm = P(\lm)Q(\lm)$ for non-constant $P$ and $Q$ in $\Z[\lm]$, then 
	\[
	p = \stern p 2 = P(2)Q(2).
	\]
	This implies $P(2) = 1$ or $Q(2) = 1$, which is impossible.
\end{proof}

\section{A Potpourri of Complex Inequalities}

This chapter is the main content of the paper.  We prove several inequalities in $\mathbb C$ needed for the main theorem.  Some results tie directly into the Stern polynomials $\stern n\lm$, while others are independent results which seem interesting in their own rights.

\subsection{Preliminaries}

Here we list some inequalities which are not directly related to complex numbers but which will be used several times in the following three sections.

Our first inequality is a lower bound on $\sin x$.  While it is not as tight
as the Taylor inequality $\sin x \geq x - \frac{x^3}6$ for small $x$, it exhibits a second equality case at $x = \pi$.

\begin{prop}\label{prop:sinc-ineq}
	For all $x \in\R$,
	\begin{equation}\label{ineq-sinc}
		\sin x \geq x - \frac{x^2}{\pi}.
	\end{equation}
\end{prop}

\begin{figure}[ht]
	\centering
	\includegraphics[scale=0.75]{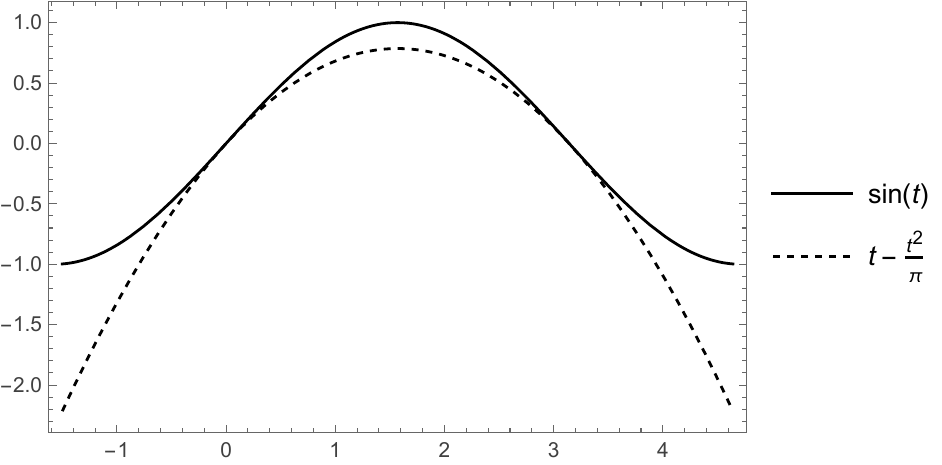}
	\caption{\small{A plot of the inequality in Proposition \ref{prop:sinc-ineq}}.}
	\label{ineq:sinc}
\end{figure}

\begin{proof}
	There are two cases to consider.
	
	\begin{itemize}
		\item First suppose $x\in[0,\pi]$.  Because both $\sin x$ and $x - \tfrac{x^2}{\pi}$ are symmetric about the axis $x = \tfrac\pi 2$, it suffices to prove the inequality for $x\in[0,\tfrac\pi 2]$.  Within this smaller interval, cosine is concave down with $\cos 0 = 1$ and $\cos\tfrac \pi 2 = 0$, so
		\begin{equation}\label{eqn:cos-ineq}
			\cos x \geq 1 - \frac{2x}{\pi}.
		\end{equation}
		Integrating both sides of \eqref{eqn:cos-ineq} yields \eqref{ineq-sinc} in this case.
		\item Now suppose $x\notin[0,\pi]$.  By symmetry, it suffices to prove the inequality for $x\geq \pi$.  In this interval, $1 - \tfrac{2x}{\pi} \leq -1 \leq \cos x$. It follows that \eqref{eqn:cos-ineq}, and thus \eqref{ineq-sinc}, holds as well.
	\end{itemize}
\end{proof}

We also need a variation of Bernoulli's Inequality which lower bounds $(1+x)^r$ by its second-order Taylor expansion.  If $r\geq 2$ is a positive integer, then certainly
\begin{equation}\label{eqn:lower-bound-taylor-series}
	(1+x)^r = \sum_{j=0}^r \binom rj x^j \geq 1 + rx + \binom r2 x^2\quad\text{whenever }x\geq 0.
\end{equation}
However, this argument does not work for general $r\in\mathbb R$, because the binomial coefficients $\binom rj$ may be negative.  Nevertheless, \eqref{eqn:lower-bound-taylor-series} still holds if $r$ is real.  We present a proof using Taylor expansions.  For a generalization of this inequality, consult
\cite{Gerber1968}.

\begin{prop}\label{prop:ineq-bernoulli}
	For all $x\geq 0$ and real $r \geq 2$, 
	\[
	(1+x)^r \geq 1 + rx + \binom r2 x^2.
	\]
\end{prop}
\begin{proof}
	Let $f(t) \coloneqq (1+t)^r$.  Observe that $f'''(t) = r(r-1)(r-2)(1+t)^{r-3}$ is nonnegative on $(0,\infty)$.  It follows from Taylor's Theorem that
	\[
	(1+x)^r = 1 + rx + \binom r2 x^2 + \int_0^x \frac{f'''(t)}6(x-t)^2\,dt \geq 1 + rx + \binom r2 x^2.\qedhere
	\]
\end{proof}

\subsection{Bounding Geometric Series on $\{\Re w \geq 1\}$}

The main technical difficulty in our analysis of the roots of $\stern n\lm$ comes from estimating $(z)_n$.  This is significantly more difficult than it seems, especially for complex numbers $z$ of the form $1 + it$ where $t$ is small.  For example, it is tempting to write
\begin{equation}\label{eqn:false-ineq}
	|(z)_n| = \abs{\frac{z^n-1}{z-1}} \geq \frac{|z|^n - 1}{|z-1|} = \frac{(1+t^2)^{n/2} - 1}t.
\end{equation}
However, the right hand side of \eqref{eqn:false-ineq} is $\tfrac n2 t + \calO(t^2)$ as $t\to 0$, whereas $|(z)_n|$ tends to $n$ as $z\to 1$.

There seem to be few attempts to strengthen \eqref{eqn:false-ineq} in the literature.  The most relevant result we could find is Problem 4795 in the American Mathematical Monthly, reproduced below.  A proof appears in \cite{AMM1962}.

\begin{thm}\label{thm:mitrinovic}
	For all $z\in\C$ with $\Re z \geq 1$, and for all $a \geq 1$,
	\[
	|(z)_a| = |1+z+\cdots + z^{a-1}| \geq |z|^{a-1}.
	\]
\end{thm}

This estimate will be useful later in this paper, but it still is not tight for $z\approx 1$.  We now focus on finding a lower bound for $|(z)_n|$ that is tight near $1$.

We first show that it suffices to lower bound $|(z)_n|$ on the line $\Re w = 1$.

\begin{lemma}\label{lemma:ineq-push-left}
	Let $n\geq 1$ be an integer, and suppose $x$, $x_0$, and $y$ are real numbers with $x\geq x_0\geq 1$.  Then
	\[
	|(x + yi)_n| \geq |(x_0 + yi)_n|.
	\]
\end{lemma}

\begin{proof}
	Both sides equal $1$ when $n = 1$, so assume $n\geq 2$.  Let $z = x + yi$ and $z_0 \coloneqq x_0 + yi$ for ease of readability.  Observe that
	\[
	|(z)_n| = |1 + z + z^2 + \cdots + z^{n-1}| = \prod_{j=1}^{n-1}|z - e^{2\pi ij/n}|
	\] 
	and
	\[
	|(z_0)_n| = |1 + z_0 + z_0^2 + \cdots + z_0^{n-1}| = \prod_{j=1}^{n-1}|z_0 - e^{2\pi ij/n}|.
	\]
	For each $j$, because $\Re(e^{2\pi ij/n}) \leq 1$ and $x\geq x_0$, we have $|z - e^{2\pi ij/n}| \geq |z_0 - e^{2\pi ij/n}|$.  Multiplying all $n-1$ such inequalities together yields
	\[
	|1 + z + z^2 + \cdots + z^{n-1}| \geq |1 + z_0 + z_0^2 + \cdots + z_0^{n-1}|.\qedhere
	\]
\end{proof}

We handle small values of $n$ numerically.  Observe that $(1+it)_n$ is a polynomial in $t$, so
\begin{align*}
	|(1+it)_n|^2 &= (1+it)_n(1-it)_n = \frac{(1+it)^n - 1}{it}\cdot \frac{(1-it)^n - 1}{-it}\\
	&= \frac{(1+t^2)^n - 2\Re(1+it)^n + 1}{t^2}
\end{align*} can be written in the form $P_n(t^2)$ for some polynomial $P_n$.  This means we may manually compute $P_n$ for small $n$ and use standard calculus techniques to numerically bound their minima.  These computations are shown in Table \ref{table:min-polys} for $1\leq n \leq 14$.

\begin{table}[!h]
	\centering
	{\fontsize{11pt}{11pt}\selectfont
		\begin{tblr}{Q[c,t]|Q[c,m]|Q[c,b]} 
			$n$ & $P_n(x)$ & $\sqrt{\min_{x\geq 0} P_n(x)}$ \\ \hline
			1 & 1 & 1\\ \hline
			2 & $x+4$ & 2 \\ \hline
			3 & $x^2+3 x+9$ & 3 \\ \hline
			4 & $x^3+4 x^2+4 x+16$ & 4 \\ \hline
			5 & $x^4+5 x^3+10 x^2+25$ & 5\\ \hline
			6 & $x^5+6 x^4+15 x^3+22 x^2-15 x+36$ & 5.8206 \\ \hline
			7 & $x^6+7 x^5+21 x^4+35 x^3+49 x^2-49 x+49$ & 6.3003 \\ \hline
			8 & $x^7+8 x^6+28 x^5+56 x^4+68 x^3+112 x^2-112 x+64$ & 6.5136\\ \hline
			9 & {$x^8+9 x^7+36 x^6+84 x^5+126 x^4$ \\ $ +108 x^3+252 x^2-216 x+81$ } & 6.5474\\ \hline
			10 & {$x^9+10 x^8+45 x^7+120 x^6+210 x^5$ \\ $+254 x^4+120 x^3+540 x^2-375 x+100$ } & 6.4727 \\ \hline
			11 & {$x^{10}+11 x^9+55 x^8+165 x^7+330 x^6$ \\ $+462 x^5+484 x^4+1089 x^2-605 x+121$} & 6.3388 \\ \hline
			12 & {$x^{11}+12 x^{10}+66 x^9+220 x^8+495 x^7+792 x^6$ \\ $+922 x^5+924 x^4-495 x^3+2068 x^2-924 x+144$ } & 6.1771\\ \hline
			13 & { $x^{12}+13 x^{11}+78 x^{10}+286 x^9 +715 x^8$ \\ $+1287 x^7+1716 x^6+1690 x^5 +1859 x^4$ \\ $-1859 x^3+3718 x^2-1352 x+169$ } & 6.0096 \\ \hline
			14 & {$x^{13}+14 x^{12}+91 x^{11}+364 x^{10}+1001 x^9$ \\ $+2002 x^8+3003 x^7 +3434 x^6+2821 x^5$ \\ $+4004 x^4  -5005 x^3+6370
				x^2-1911 x+196$ } & 5.8386
		\end{tblr}
	}
	\caption{\small{Polynomials $P_n(x)$ as well as their minima on $[0,\infty)$.}}
	\label{table:min-polys}
\end{table}

Unfortunately, these explicit polynomials are hard to analyze for large $n$.  (In particular, observe that the sequence of minima is not monotonic, instead peaking at $n = 9$.)  This means we need a fundamentally different approach to bound $P_n$ in general.

The key estimate for us is the following inequality.

\begin{thm}\label{thm:MO-ineq}
	For all real numbers $t$ and all positive integers $n$,
	\begin{equation}\label{eqn:tao-ineq}
		\abs{\left(1+\frac{it}n\right)^n - 1} \geq \big|e^{it} - 1\big| = 2\sin\frac t2.
	\end{equation}
\end{thm}

Carlo Beenakker conjectured Theorem \ref{thm:MO-ineq} in response to a question by the author on MathOverflow \cite{tao-MO}, and Terence Tao later confirmed this conjecture in the same thread. To make this paper self-contained, and to record the proof of this inequality in print, we replicate Tao's argument below with some modifications in exposition.

\begin{proof}[Proof (Tao, \cite{tao-MO})]
	The proof proceeds in four steps.  The first step rewrites \eqref{eqn:tao-ineq} into polar coordinate form to obtain the equivalent inequality \eqref{eqn:main-ineq}.  The next three steps prove \eqref{eqn:main-ineq} by performing casework on different ranges of $t$.
	
	\paragraph{\textbf{Step 1: Setup.}} For $n = 1$, observe that 
	\begin{equation}\label{eqn:case-n-1}
		|e^{it} - 1| = 2\left|\sin \frac t2\right| \leq 2\cdot \frac {|t|}2 = |t|.
	\end{equation}
	For $n = 2$, Theorem \ref{thm:MO-ineq} holds by the calculation
	\[
	\abs{\left(1 + \frac{it}2\right)^2 - 1} = \sqrt{t^2 + \left(\frac{t^2}4\right)^2} \geq |t|\geq |e^{it} - 1|,
	\]
	where the last inequality is due to \eqref{eqn:case-n-1}.  (This shows that, interestingly, the left hand side of \eqref{eqn:tao-ineq} is \textit{not} always decreasing in $n$, despite converging to $|e^{it} - 1|$ in the limit as $n\to\infty$.) In what follows, assume $n\geq 3$.
	
	\par Without loss of generality assume $t\geq 0$.  We may take advantage of the estimate $(1 + \tfrac{it}{n})^n\approx e^{it}$ by writing $(1+\tfrac{it}n)^n = re^{i(t-\eps)}$, where 
	\[
	r\coloneqq \left(1 + \frac{t^2}{n^2}\right)^{n/2}\quad\text{and}\quad \eps \coloneqq t - n\arctan \frac tn.
	\]
	We expect that $\eps$ should be relatively small.  Indeed, upon integrating the inequalities $1-x^2 \leq \tfrac{1}{1+x^2}\leq 1$ (true for all $x$) we have
	$x - \frac{x^3}3 \leq \arctan x \leq x$ for all $x\geq 0$.  It follows that
	\begin{equation}\label{eqn:bound-eps}
		0\leq \eps \leq \frac{t^3}{3n^2}.
	\end{equation}
	
	\par To finish the setup, observe that squaring both sides of \eqref{eqn:tao-ineq} yields the equivalent inequality
	\[
	r^2 - 2r\cos(t - \eps) + 1 \geq 2 - 2\cos t.
	\]
	This rearranges to 
	\begin{equation}\label{eqn:main-ineq}
		(r-1)^2 + 2(r-1)\Big(1 - \cos(t-\eps)\Big) \geq 2\Big(\cos(t-\eps) - \cos t\Big).
	\end{equation}
	The rest of the proof will establish \eqref{eqn:main-ineq} by splitting into cases based on the value of $t$.
	
	\paragraph{\textbf{Step 2: $\boldsymbol{t > \tfrac 83}$.}} In this region, $r-1$ is large enough that we expect the first term $(r-1)^2$ to dominate on its own.  By Bernoulli's Inequality, we make the estimate
	\[
	r - 1 = \left(1 + \frac{t^2}{n^2}\right)^{n/2} - 1 \geq \frac n2 \cdot \frac{t^2}{n^2} = \frac{t^2}{2n}.
	\]
	Furthermore, applying the Mean Value Theorem to $x\mapsto \cos x$ yields
	\[
	\cos(t-\eps) - \cos t \leq \eps \leq \frac{t^3}{3n^2}.
	\]
	It follows that $(r-1)^2 \geq 2(\cos(t-\eps) - \cos t)$ when
	\[
	\left(\frac{t^2}{2n}\right)^2 \geq 2\cdot \frac{t^3}{3n^2},
	\]
	or, equivalently, when $t > \frac 83$.
	
	\paragraph{\textbf{Step 3: $\boldsymbol{\tfrac \pi 2 < t \leq \tfrac 83}$.}}
	Within this range, we instead show that the second term, namely $2(r-1)(1-\cos(t-\eps))$, dominates.  Here, $\cos t\leq 0$ and
	\[
	1 - \cos(t-\eps) \geq 1 - \eps \geq 1 - \frac{t^3}{3n^2}.
	\] 
	It follows that \eqref{eqn:main-ineq} will be satisfied (just using the second term) if
	\[
	2\cdot \frac n2 \cdot \frac{t^2}{n^2}\left(1 - \frac{t^3}{3n^2}\right) \geq 2\frac{t^3}{3n^2},
	\]
	which simplifies to
	\[
	\frac 23t + \frac{t^3}{3n^2} \leq n.
	\]
	It remains to check this inequality for $n\geq 3$ and $t\leq \tfrac 83$.  This follows from the computation
	\[
	\frac 23t + \frac{t^3}{3n^2} \leq \frac 23\cdot \frac 83 + \frac{(8/3)^3}{3\cdot 3^2} = \frac{1808}{729} < 3.
	\]
	
	\paragraph{\textbf{Step 4: $\boldsymbol{0\leq t \leq \tfrac \pi 2}$.}} We finish with the most delicate case of the proof.  Rewrite \eqref{eqn:main-ineq} slightly as
	\begin{equation}\label{eqn:modified}
		(r-1)^2 + 2(r-1)(1-\cos t) \geq 2r\Big(\cos(t-\eps) - \cos t\Big).
	\end{equation}
	We do this to take advantage of concavity of cosine in this region.  Indeed, because $t$ and $t-\eps = n\arctan \tfrac tn$ are both in the interval $[0,\tfrac \pi 2]$, 
	\[
	\cos(t-\eps) - \cos t \leq \eps\sin t.
	\]
	This yields a crucial extra power of $t$ in the limit as $t\to 0^+$.
	
	\par Now observe that
	\[
	\frac{1 - \cos t}{\sin t} = \tan \frac t2 \geq \frac t2.
	\]
	Thus, after dividing through by $\sin t$, it suffices to establish the bound
	\[
	2(r-1) \frac t2 \geq 2r\eps
	\]
	(again taking only the second term in the left hand side of \eqref{eqn:modified}).  From our previous inequalities it would suffice to show that
	\[
	2\cdot \frac n2 \cdot \frac{t^2}{n^2}\cdot \frac t2 \geq 2r\frac{t^3}{3n^2},
	\]
	or $r\leq \frac 34n$.  Magically, all powers of $t$ have canceled out.  
	
	\par To finish this case, we make the bound
	\[
	r\leq \left(1 + \frac{(\pi/2)^2}{n^2}\right)^{n/2} \leq \exp\left(\frac{(\pi/2)^2}{n^2}\frac n2\right) = \exp\left(\frac{\pi^2}{8n}\right).
	\]
	It remains to show $\exp(\tfrac{\pi^2}{8n})\leq \tfrac 34n$ for all $n\geq 3$.  Indeed, the left-hand side is decreasing in $n$ while the right-hand side is increasing, and at $n = 3$ the inequality becomes
	\[
	\exp\left(\frac{\pi^2}{24}\right) \leq \frac 94.
	\]
	A quick calculation reveals 
	\[
	\exp\left(\frac{\pi^2}{24}\right) < \exp\left(\frac12\right) = \sqrt{e} < 2 < \frac{9}{4},
	\]
	so $\exp(\tfrac{\pi^2}{8n}) \leq \tfrac 34n$ holds for all $n\geq 3$.
\end{proof}

\subsection{Bounding Geometric Series on $\calB$}

Recall that $\calB$ is the disk $\{|w-2| \leq 1\}\subseteq\C$.  While Theorem \ref{thm:MO-ineq} is strong in its own right, it is not quite sufficient to show
Theorem \ref{thm:prime-poly-irred}.  We opt to establish a separate lower bound for $|(w)_n|$
when $w\in\calB$.  This bound is not the strongest possible (see the remark before Proposition \ref{prop:ineq-smallcases}), but
it is sufficient for our purposes.

\begin{thm}\label{thm:min-geom-series}
	Let $n$ be a positive integer, and suppose $z\in\calB$.  Then $|(z)_n| \geq \min(n,\tfrac{11}2)$.
\end{thm}

For all $n\in\mathbb N$ and $t\in(0,n\pi]$, let 
\[
w_{n,t} \coloneqq \overline{2 - e^{it/n}} = \left(2-\cos\frac{ t}n\right) + i \sin\frac{t}n.
\]  
Just as in the proof of Theorem \ref{thm:MO-ineq}, we aim to take advantage of the approximation $w_{n,t}\approx e^{it}$ for large $n$.  The next few propositions will establish several lower bounds for $|(w_{n,t})_n|$.

\begin{prop}\label{prop:linear-near-zero}
	For integer $n\geq 1$ and real $t\geq 0$, we have
	\[
	\abs{\frac{w_{n,t}^n - 1}{w_{n,t} - 1}} \geq n\left(1 - \frac t{2\pi}\right).
	\]
\end{prop}

\begin{proof}
	By Lemma \ref{lemma:ineq-push-left}, because $\Im w_{n,t} = \sin\tfrac tn$ and $\Re w_{n,t} \geq 1$, we know
	\[
	\abs{\frac{w_{n,t}^n - 1}{w_{n,t} - 1}} \geq \abs{\frac{(1 + i\sin \frac tn)^n - 1}{(1 + i\sin \frac tn) - 1}} = \frac{\abs{(1 + i\sin \frac tn)^n - 1}}{\sin\frac tn}.
	\]
	Now write $\sin \frac tn = \tfrac 1n\cdot n\sin\tfrac tn$ and apply Lemma \ref{thm:MO-ineq} to obtain 
	\begin{align*}
		\abs{\left(1 + i\sin \frac tn\right)^n - 1} &= \abs{\left(1 + \frac in\cdot n\sin \frac tn\right)^n-1} \\
		&\geq \abs{e^{in\sin\frac tn} - 1} = 2\sin\left(\frac n2\sin\frac tn\right).
	\end{align*}
	Finally, using the inequalities $\tfrac{\sin x}x \geq 1 - \frac x\pi$ and $\sin x \leq x$, valid for $x > 0$, we obtain 
	\begin{align*}
		\frac{\abs{(1 + i\sin \frac tn)^n - 1}}{\sin\frac tn} &\geq \frac{2\sin(\frac n2\sin\frac tn)}{\sin\frac tn}
		= n\cdot\frac{\sin(\frac n2\sin \frac tn)}{\frac n2\sin\frac tn}\\
		&\geq n\left(1 - \frac{\frac n2\sin\frac tn}{\pi}\right)
		\geq n\left(1 - \frac t{2\pi}\right). \qedhere
	\end{align*}
\end{proof}

The second bound concerns values of $t$ away from $0$.

\begin{prop}\label{prop:simple-ineq}
	For all $n\geq 4$ and $t\geq 0$, we have
	\[
	\abs{\frac{w_{n,t}^n - 1}{w_{n,t} - 1}} \geq 2n\sin\frac t{2n}.
	\]
\end{prop}

\begin{proof}
	Because $w_{n,t} = (2-\cos\tfrac tn) + i\sin\tfrac tn$, compute
	\begin{align*}
		\abs{w_n}^2 &= \left(2 - \cos\frac tn\right)^2 + \sin^2\frac tn\\
		&= 5 - 4\cos\frac tn
		= 1 + 8\sin^2 \frac{t}{2n}.
	\end{align*}
	It follows by the Triangle Inequality and $\sin x \geq x - \tfrac{x^3}6$ that
	\begin{align*}
		\abs{\frac{w_{n,t}^n - 1}{w_{n,t} - 1}} &\geq \frac{\abs{w_{n,t}}^n - 1}{\abs{w_{n,t} - 1}} 
		= \frac{(1 + 8\sin^2\frac t{2n})^{n/2} - 1}{2\sin\frac t{2n}}.
	\end{align*}
	Finally, Bernoulli's Inequality yields
	\[
	\frac{(1 + 8\sin^2\frac t{2n})^{n/2} - 1}{2\sin\frac t{2n}} \geq
	\frac{\big(1 + 4n\sin^2\frac t{2n}\big) - 1}{2\sin\frac t{2n}} = 2n\sin\frac t{2n}. \qedhere
	\]
\end{proof}
Our third inequality is a slight tightening of Proposition \ref{prop:simple-ineq} as $t\to 0$.

\begin{prop}\label{prop:tighter-ineq}
	For all $n$ and $t\in[0,2\pi]$, we have
	\begin{equation}\label{eqn:lower-bound-large-t}
		\abs{\frac{w_{n,t}^n - 1}{w_{n,t} - 1}} \geq t + t^3\left(\frac{1}{2n} - \frac{25}{24n^2}\right) - t^5\left(\frac{1}{8n^3} - \frac{1}{16n^4}\right).
	\end{equation}
\end{prop}

The right hand side of \eqref{eqn:lower-bound-large-t} is the fifth-order Taylor approximation of the left hand side as a function of $t$.  The simpler inequality $|w_{n,t}^n - 1| \geq t|w_{n,t} - 1|$ is easier to prove, but the resulting bounds are not strong enough to prove Theorem \ref{thm:min-geom-series} without quadrupling the number of base cases.

\begin{proof}
	As before, derive the inequality
	\[
	\abs{\frac{w_{n,t}^n - 1}{w_{n,t} - 1}} \geq
	\frac{(1 + 8\sin^2\frac t{2n})^{n/2} - 1}{2\sin\frac t{2n}}.
	\]
	This time, we use Proposition \ref{prop:ineq-bernoulli} in conjunction with $\sin x \geq x - \tfrac{x^3}6$ to obtain
	\begin{align*}
		\frac{(1 + 8\sin^2\frac t{2n})^{n/2} - 1}{2\sin\frac t{2n}} &\geq \frac{(1 + 4n\sin^2\frac{t}{2n} + 8n(n-2)\sin^4\frac{t}{2n}) - 1}{2\sin\frac{t}{2n}}\\
		&= 2n\sin\frac{t}{2n} + 4n(n-2)\sin^3\frac{t}{2n}\\
		&\geq 2n\left(\frac{t}{2n} - \frac{t^3}{48n^3}\right) + 4n(n-2)\left(\frac{t}{2n} - \frac{t^3}{48n^3}\right)^3\\
		&=t - \frac{t^3}{24n^2} + \frac{(n-2)t^3}{2n^2}\left(1 - \frac{t^2}{24n^2}\right)^3.
	\end{align*}
	Finally, Bernoulli's Inequality allows us to bound the nonlinear terms:
	\begin{align*}
		t - \frac{t^3}{24n^2} + \frac{(n-2)t^3}{2n^2}\left(1 - \frac{t^2}{24n^2}\right)^3
		&\geq t - \frac{t^3}{24n^2} + \frac{(n-2)t^3}{2n^2}\left(1 - \frac{t^2}{8n^2}\right)\\
		&= t + t^3\left(\frac{1}{2n} - \frac{25}{24n^2}\right) - t^5\left(\frac{1}{8n^3} - \frac{1}{16n^4}\right).\qedhere
	\end{align*}
\end{proof}

We are now able to prove Theorem \ref{thm:min-geom-series}.

\begin{proof} Table \ref{table:min-polys} shows Theorem \ref{thm:min-geom-series} for $n\leq 14$, so assume $n\geq 15$.  
	
	\par First suppose $t\geq 2\pi$.  Here, use Proposition \ref{prop:simple-ineq} to write
	\[
	\abs{(w_{n,t})_n} \geq 2n\sin\frac{t}{2n} \geq 2n\sin\frac \pi n.
	\]
	The right hand side is increasing as a function of $n$ and evaluates to $12\sin\frac\pi 6 = 6$ for $n = 6$.  It follows that $\abs{(w_{n,t})_n}\geq 6$ for $n\geq 6$, ergo for $n\geq 15$.  
	
	\par We may now assume $t\in[0,2\pi]$.  By Propositions \ref{prop:linear-near-zero} and \ref{prop:simple-ineq}, we know
	\begin{equation}\label{eqn:complicated-ineq}
		\abs{(w_{n,t})_n}\geq \max\left\{n\left(1 - \frac{t}{2\pi}\right),t + t^3\left(\frac{1}{2n} - \frac{25}{24n^2}\right) - t^5\left(\frac{1}{8n^3} - \frac{1}{16n^4}\right)\right\}.
	\end{equation}
	
	\par To make the analysis simpler, we bound the $t^5$ term in \eqref{eqn:complicated-ineq}.  Because $t\in[0,2\pi]$,
	\begin{align*}
		t^5\left(\frac{1}{8n^3} - \frac{1}{16n^4}\right) &= \frac{t^3}{n^2}\cdot t^2\left(\frac{1}{8n} - \frac{1}{16n^2}\right)\\
		& \leq \frac{t^3}{n^2}\cdot (2\pi)^2\left(\frac{1}{8\cdot 15} - \frac{1}{16\cdot 15^2}\right) = \frac{29\pi^2}{900}\frac{t^3}{n^2} < \frac{t^3}{3n^2}.
	\end{align*}
	Combining this with Propositions \ref{prop:linear-near-zero} and \ref{prop:simple-ineq}, it follows that
	\begin{align}
		\abs{(w_{n,t})_n}&\geq \max\left\{n\left(1 - \frac t{2\pi}\right), t + t^3\left(\frac{1}{2n} - \frac{25}{24n^2}\right) - \frac{t^3}{3n^2}\right\} \nonumber\\
		&= \max\left\{n\left(1 - \frac t{2\pi}\right), t + t^3\left(\frac{1}{2n} - \frac{11}{8n^2}\right)\right\}.\label{eqn:bounding-max}
	\end{align}
	
	\begin{figure}[ht]
		\centering
		\includegraphics[scale=0.5]{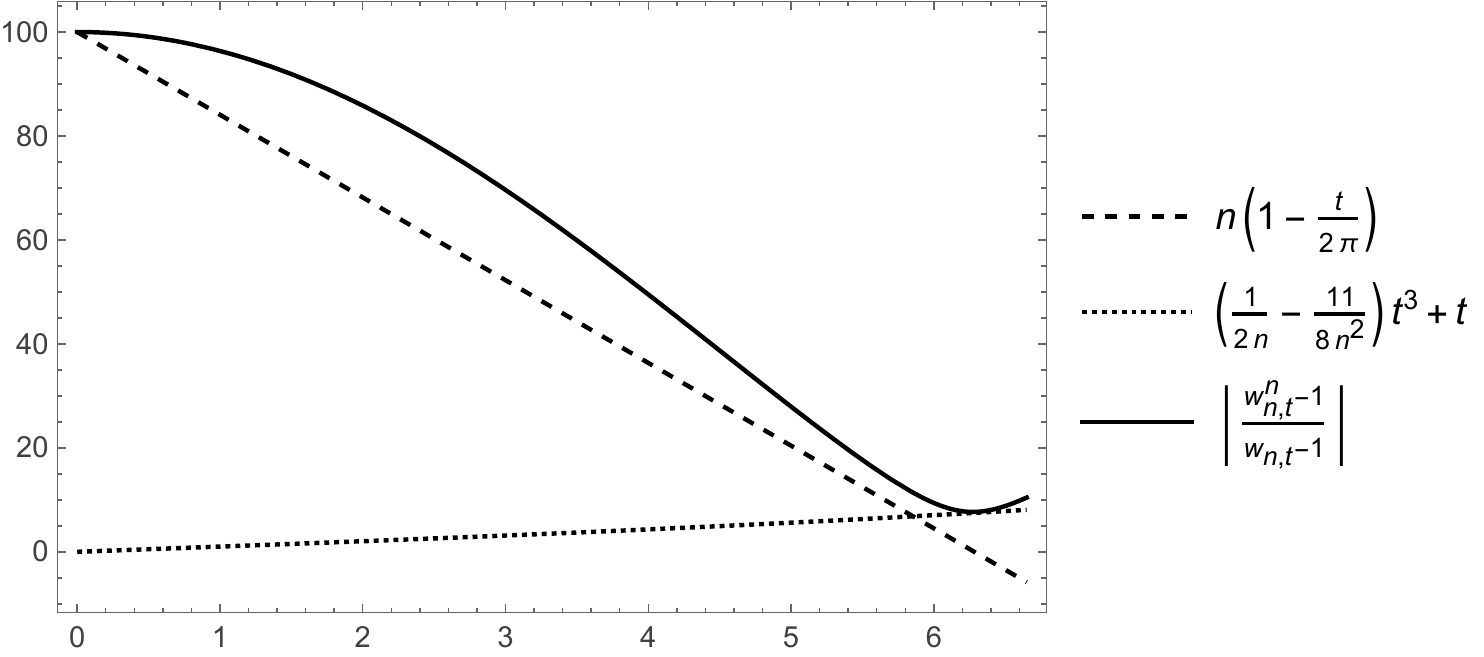}
		\caption{\small{Plot comparing $|(w_{n,t})_n|$ to the two expressions in \eqref{eqn:bounding-max} when $n = 100$.  Note the near-equality case around $t = 2\pi$.}}
	\end{figure}
	
	Now let
	\[
	M\coloneqq 2\pi\left(1 - \frac{5.5}{n}\right) = \pi\left(2 - \frac{11}n\right).
	\]
	There are two cases to consider.  First, suppose $t\leq M$.  Then
	\[
	n\left(1 - \frac{t}{2\pi}\right) \geq n\left(1 - \frac{M}{2\pi}\right) = \frac{11}2.
	\]
	Now suppose $t\geq M$.  Because $n\geq 15$, the coefficient $\tfrac{1}{2n} - \tfrac{11}{8n^2}$ is positive, so
	\begin{equation}\label{eqn:asymp-w-tn}
		t + t^3\left(\frac{1}{2n} - \frac{11}{8n^2}\right) \geq \pi\left(2 - \frac{11}n\right) + \pi^3\left(2 - \frac{11}n\right)^3\left(\frac{1}{2n} - \frac{11}{8n^2}\right).
	\end{equation}
	To estimate \eqref{eqn:asymp-w-tn}, let $x = \tfrac 1n$, so it suffices to analyze the polynomial
	\[
	P(x) := \pi(2-11x) + \pi^3(2-11x)^3(\tfrac 12x - \tfrac{11}{8}x^2).
	\]
	The polynomial $P$ has exactly one critical point in the interval $[0,\tfrac 1{15}]$, occurring at $x = x_0 \approx 0.024$.  In particular, $P$ is increasing on $[0,x_0]$ and decreasing on $[x_0,\tfrac1{15}]$.  Compute
	\[
	P(0) = 2\pi\quad\text{and}\quad P\left(\frac 1{15}\right) = \frac{19 \pi }{15}+\frac{336091 \pi ^3}{6075000} \approx 5.695 > \frac{11}2.
	\]
	It follows that $P(x) > \tfrac{11}2$ for all $x\in[0,\tfrac{1}{15}]$, and thus $|w_{n,t}^n - 1| \geq M|w_{n,t}-1|$ for all $t\in[0,2\pi]$ and $n\geq 15$.  
\end{proof}

\begin{remark}
	\leavevmode
	\begin{enumerate}
		\item In the above proof, the maximum value of $P(x)$ in the interval $[0,\tfrac{1}{15}]$ is $P(x_0)\approx 7.2723$. 
		\item We conjecture the stronger bound $|(z)_n| \geq \min(n,2\pi)$.  The methods above cannot extend to this tougher inequality without extensive computer assistance.
	\end{enumerate}
\end{remark}

\par Finally, we will need two more special cases.  Our previous estimates are not \textit{quite} strong enough to obtain these inequalities, so we need to prove them separately.

\begin{prop}\label{prop:ineq-smallcases}
	Suppose $z\in \calB$.  Then
	\begin{align}
		|(z)_5| &= |1+z+z^2+z^3+z^4| \geq 5|z|,\label{eqn:min-5}\\
		|(z)_6| &= |1+z+z^2+z^3+z^4+z^5| \geq 6|z|.\label{eqn:min-6}
	\end{align}
\end{prop}

\begin{proof}
	We first prove \eqref{eqn:min-5}.  Observe that the function $f(z) = \tfrac{z}{(z)_5}$ is analytic in $\calB$, so by the Maximum Modulus Principle we may assume $z\in\partial B$, i.e. $|z-2| = 1$.  Write $z = re^{i\theta}$.  Rewrite the given inequality to $|z^5 - 1| \geq 5|z| |z-1|$ and square both sides; this yields the equivalent inequality
	\begin{equation}\label{eqn:polar-coord}
		r^{10} - 2r^5\cos(5\theta) + 1 \geq 25r^2(r^2 - 2r\cos\theta + 1).
	\end{equation}
	Observe that $|z-2| = 1$ implies $r^2 - 4r\cos\theta + 4 = 1$, or $\cos\theta = \frac{r^2+3}{4r}$.  Under this substitution, the right hand side of \eqref{eqn:polar-coord} becomes
	\[
	25r^2\left[r^2 - 2r\cdot \frac{r^2+3}{4r} + 1\right] = \frac{25}2(r^4 - r^2).
	\]
	Analogously, the left hand side of \eqref{eqn:polar-coord} becomes
	\begin{align*}
		r^{10} - 2r^5\cos(5\theta) + 1 &= r^{10} - 2r^5\left(16\cos^5\theta - 20\cos^3\theta + 5\cos\theta\right) + 1\\
		&= r^{10} - 2r^5\left[16\left(\frac{r^2+3}{4r}\right)^5 - 20\left(\frac{r^2+3}{4r}\right)^3 + 5\left(\frac{r^2+3}{4r}\right)\right] + 1\\
		&= r^{10} - 32\left(\frac{r^2+3}4\right)^5 - 20r^2\left(\frac{r^2+3}{4}\right)^3 + 5r^4\left(\frac{r^2+3}4\right)+1\\
		&= \frac1{32}(31r^{10} + 5r^8 + 10r^6 + 30r^4 + 135r^2 - 211).
	\end{align*}
	
	It follows that \eqref{eqn:polar-coord} is equivalent to the inequality 
	\begin{equation}\label{eqn:reduction-inequality}
		\frac1{32}(31r^{10} + 5r^8 + 10r^6 + 30r^4 + 135r^2 - 211) \geq \frac{25}2(r^4 - r^2).
	\end{equation}
	But \eqref{eqn:reduction-inequality} follows from the unexpected factorization
	\[
	\text{LHS} - \text{RHS} = \frac{(r^2-1)^3(31r^4 + 98r^2 + 211)}{32} \geq 0.
	\]
	This proves \eqref{eqn:min-5}.
	
	The proof of \eqref{eqn:min-6} is similar, so we only sketch the details.  By the Maximum Modulus Principle we may assume $|z-2| = 1$.  Squaring both sides yields the equivalent inequality
	\begin{equation}\label{eqn:polar-coord-2}
		r^{12} - 2r^6\cos(6\theta) + 1 \geq 36r^2(r^2 - 2r\cos\theta + 1).
	\end{equation}
	Substituting $\cos \theta = \tfrac{r^2+3}{4r}$ and expanding yields the inequality
	\[
	\frac{1}{64} \left(63 r^{12}+6 r^{10}+9 r^8+20 r^6+81 r^4+486 r^2-665\right) \geq 18(r^4 - r^2).
	\] 
	Here, we have the factorization
	\begin{equation}\label{eqn:ineq-2}
		\text{LHS} - \text{RHS} =  \frac{(r^2-1)^2(63r^8+132r^6+210r^4+308r^2-665)}{64}.
	\end{equation}
	The polynomial $g(x)\coloneqq  63x^8+132x^6+210x^4+308x^2-665$ is increasing on $[0,\infty)$ and $g(1) = 48$.  Because $r\in[1,3]$, we know $g(r) \geq 0$, and thus the right hand side of \eqref{eqn:ineq-2} is always nonnegative.
\end{proof}

\begin{remark}
	One might suspect that $|(z)_n| \geq n |z|$ more generally whenever $z\in\calB$.  However, numerical evidence suggests that the value $\min_{z\in B}\abs{\tfrac{(z)_n}{z}}$ approaches $2\pi$ as $n\to\infty$.  It seems that this inequality is true \textit{only} for $n = 5$, $6$, and $7$, but we do not have a proof.
\end{remark}

\subsection{Set Inclusion Inequalities}

This section focuses on showing that certain sets in $\C$ are subsets of $E = E_{\pi/12}$, where we use the definition of $E_\alpha$ given in Theorem \ref{thm:parabola}.  

\par We first establish two preliminary facts about elements of $\calB$. Because nonreal zeros of $\stern n\lm$ come in conjugate pairs, it suffices to examine $z\in\calB^+$.

\begin{prop}\label{prop:silli}
	Let $z\in\calB^+$.  Then $0\leq \arg z \leq \tfrac\pi 6$ and
	\begin{equation}\label{eqn:re-ineq}
		|\Re z^{-4}|\leq |\Re z^{-2}|.
	\end{equation}
\end{prop}

\begin{proof}
	The lower bound $0\leq \arg z$ is clear because $z$ lies in the first quadrant.  For the upper bound, observe that the line $y = \tfrac{1}{\sqrt 3}x$ is tangent to the circle $(x-2)^2+y^2 = 1$, so $\arg z\leq \arctan \tfrac{1}{\sqrt 3} = \tfrac\pi 6$.
	
	\par To prove \eqref{eqn:re-ineq}, we claim that
	\begin{equation}\label{eqn:cos-ineq-2}
		\abs{\cos 4\theta} \leq \abs{\cos 2\theta} \quad\text{whenever}\quad \theta\in[0,\tfrac\pi 6].
	\end{equation}
	To prove this, note that the function $h(x) = |2x^2 - 1|$ is decreasing for $x\in[\tfrac12,\tfrac{\sqrt 2}2]$ and increasing for $x \in[\tfrac{\sqrt 2}2, 1]$.  Furthermore, $h(\tfrac12) = \tfrac12$ and $h(1) = 1$.  It follows that 
	\[
	|2x^2 - 1|\leq |x|\quad\text{for all }x\in[\tfrac12,1],
	\]
	which is equivalent to \eqref{eqn:cos-ineq-2} under the substitution $x = \cos 2\theta$.
	
	\par This implies \eqref{eqn:re-ineq}, as
	\[
	\abs{\Re z^{-4}} = |z|^{-4}\abs{\cos 4\theta} \leq |z|^{-2}|\cos 2\theta| = |\Re z^{-2}|. \qedhere
	\]
\end{proof}

We now turn to the main results of this section.

\begin{prop}\label{prop:sets-parabola}
	The following sets in $\C$ are subsets of $E$:
	
	\noindent
	\def\arraystretch{1.25}
	\setlength{\tabcolsep}{25pt}
	\begin{tabular}{ll}
		(a) $\calE_a \coloneqq \ball{0}{0.23}$, & 
		(b) $\calE_b\coloneqq \ball{\frac32}{7/4}^+$, \\
		(c) $\calE_c\coloneqq \ball{0}{1/5} + \ball{2}{1}^+$, &
		(d) $\calE_d\coloneqq \{\Re w \geq 1.15\}^{-1} = \ball{0.575}{0.575}$, \\
		(e) $\calE_e\coloneqq \ball{0.25}{0.35}$. &
	\end{tabular}
\end{prop}

See Figure \ref{fig:sets-parabola} for plots of $\calE_a$ through $\calE_e$. The upshot of this proposition is that it will allow us to show that a complex number $w$ is in $E$ by instead showing it lies in one of the sets $\calE_a$ through $\calE_e$.  These new subsets appear more naturally in our calculations.

\begin{figure}[ht]
	\begin{center}
		\includegraphics[scale=0.78]{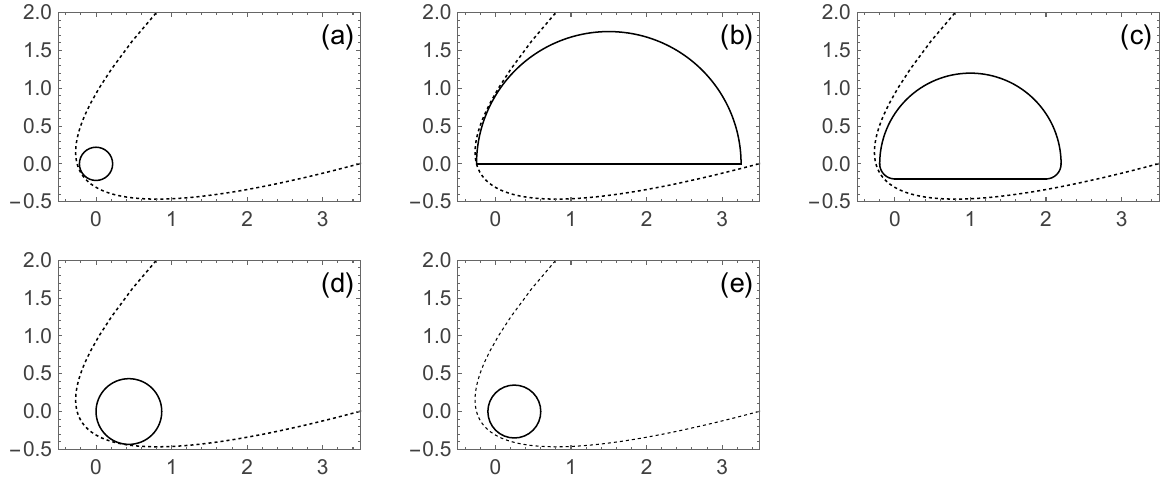}
	\end{center}
	\caption{\small{Plots of each of the five sets in Proposition \ref{prop:sets-parabola} (solid) compared with $E$ (dashed). Some of the bounds are quite tight.}}
	\label{fig:sets-parabola}
\end{figure}

\begin{proof}
	It is known (e.g. Chapter 3 Exercise 26 of \cite{Lorentzen2008}) that $E_\alpha$ can be written in the form
	\[
	\left\{
	r e^{i(\theta + 2\alpha)}: 0\leq r\leq \frac{\frac12\cos^2\alpha}{1-\cos\theta}
	\right\}.
	\]
	In our case, $\alpha = \tfrac\pi{12}$, so
	\[
	E_{\pi/12} = \left\{r e^{i(\theta + \pi/6)}:0\leq r\leq \frac\lm{1-\cos\theta}\right\}\quad\text{for}\quad \lm = \frac{2+\sqrt 3}8.
	\]
	We now proceed with the bounding.  
	
	\par For part (a), observe that $1-\cos\theta \geq 2$, so along $\partial E$ the minimum value of $r$ is $\tfrac12\lm > 0.233$.  This implies $\calE_a\subseteq E$.
	
	For part (b), again let $z$ lie on $\partial E$.  The restriction $\Im z \geq 0$ implies $\theta\in[-\tfrac\pi 6, \tfrac{5\pi}6]$.  Compute
	\begin{align*}
		\abs{z-\tfrac32}^2 &= r^2 - 3r\cos\left(\theta + \frac\pi 6\right) + \frac 94\\
		&= \frac{\lm^2}{(1-\cos\theta)^2} - \frac{3\lm \cos(\theta + \frac\pi 6)}{1-\cos\theta} + \frac 94 =: g(\theta).
	\end{align*}
	It remains to analyze the behavior of $g$.  Although $g$ is increasing on $[-\tfrac\pi 6,0)$, there are two critical points in $(0,\tfrac{5\pi}6]$: one at $\theta\approx 1.11731$ and one at $\theta\approx 1.75855$.  Checking these points, as well as the endpoints, yields that the minimum of $g$ on $[-\tfrac\pi 6, \tfrac{5\pi}6]$ is $g(\tfrac{5\pi}6) = \tfrac{49}{16}$.  Thus $|w-\tfrac32| \geq \tfrac{7}{4}$ and $\calE_b\subseteq E$.
	
	For part (c), we remark that
	\[
	\calE_c \subset C_1\cup C_2\cup C_3,
	\]
	where 
	\[
	C_1 = \ball{1}{6/5}^+,\quad C_2 = \ball{0}{1/5},\quad\text{and}\quad C_3 = [0,\tfrac{11}5]\times[-\tfrac15,0].
	\]
	(See Figure \ref{fig:sets-parabola}(c).)  By part (a), $C_1\subseteq \calE_a\subseteq E$, and by part (b), $C_2\subseteq \calE_b\subseteq E$.  To show $C_3\subseteq E$, observe that $E$ is the convex hull of the four complex numbers $0$, $-\tfrac 15i$, $\tfrac {11}5$, and $\tfrac{11}5 - \tfrac 15i$.  Each of these four points lies in $E$ by simple computation, and $E$ is a convex set.  It follows that $C_3\subseteq E$, and so $\calE_c\subseteq E$.
	
	For part (d), observe that 
	\[
	E^{-1} = \left\{r e^{i(\theta - \pi/6)}:r\geq \frac 1\lm(1-\cos\theta)\right\}.
	\]
	Thus, writing $w = re^{i(\theta - \pi/6)} \in\partial E$ in polar coordinates,
	\begin{equation}\label{eqn:part-d}
	\Re w = \frac 1\lm (1-\cos\theta)\cos(\theta - \tfrac\pi 6).
	\end{equation}
	The maximum value of \eqref{eqn:part-d} on $[0, 2\pi]$ is $1.13861\ldots < 1.14$, occurring when $\theta = \tfrac{4\pi}9$.
	It follows that $\{\Re w \geq 1.15\}\subseteq E^{-1}$, so $\calE_d\subseteq E$.
	
	Finally, for part (e), we proceed as we did in our proof of (b).  In this case, our goal is to minimize the quantity
	\begin{align*}
		\abs{z-\tfrac14}^2 &= r^2 - \frac 12r\cos\left(\theta + \frac\pi 6\right) + \frac 1{16}\\
		&= \frac{\lm^2}{(1-\cos\theta)^2} - \frac{\lm \cos(\theta + \frac\pi 6)}{2(1-\cos\theta)} + \frac 1{16} =: h(\theta).
	\end{align*}
	Within the interval $[0,2\pi]$, $h$ has exactly one local minimum at $\theta =: \theta_0 \approx 4.46993$.  It follows that $h(\theta) \geq h(\theta_0) > 0.15$, and so $|z-\tfrac14| > \sqrt{0.15} \geq 0.387 > 0.35$.  In turn, $\calE_e\subseteq E$.
\end{proof}

\section{Proof of Theorem \ref{thm:criterion-nonzero}}

We are finally ready to prove Theorem \ref{thm:criterion-nonzero}.

\begin{proof}[Proof (Theorem \ref{thm:criterion-nonzero})]
	Our proof will involve splitting the ordered pairs $(a,b)\in\N^2$ into several cases and analyzing each case independently.  In some cases, we will analyze the original ratio $z_{a,b} = \tfrac{z^a}{(z)_a(z)_b}$, while in others we instead examine its reciprocal $z_{a,b}^{-1} = \tfrac{(z)_a(z)_b}{z^a}$.  We shall refer to these expressions by $(*)$ and $(*^{-1})$, respectively.
	
	\begin{itemize}
		\item $\boldsymbol{b=1,a\leq 4{:}}$ Begin with $(*^{-1})$.  The quotient expands to
		\[
		\frac{(z)_a}{z^a} = \frac{1+z+\cdots + z^{a-1}}{z^a} = \frac 1z + \frac1{z^2} + \cdots + \frac{1}{z^a}.
		\]
		Note that $\Im z^{-j} \leq 0$ for $1\leq j\leq a$.  Furthermore, $\Re z^{-j} \geq 0$ for $j=1,2,3$, and $\Re z^{-4} \geq \Re z^{-2}$ by
		part 2 of Proposition \ref{prop:silli}.  It  follows that
		\[
		\Re \frac{(z)_a}{z^a} \geq \Re \frac 1z \geq \frac 13.
		\]
		Combining both parts yields $\tfrac{z^a}{(z)_a}\in \ball{\tfrac 32}{3/2}^+\subseteq \calE_b\subseteq E$.
		
		\item $\boldsymbol{2 \leq b\leq 4, a \leq 4{:}}$ Begin with $(*^{-1})$.  We perform similar computations to the previous part, but we can no longer assert that $\tfrac{(z)_b(z)_a}{z^a}$ lies in the lower half plane.  Write
		\begin{align*}
			\frac{(z)_b(z)_a}{z^a} &= (1+z+\cdots+z^{b-1})\left(\frac 1z + \cdots + \frac{1}{z^a}\right)\\
			&=\sum_{\substack{0\leq j\leq b-1\\1\leq k\leq a}}z^{j-k} = \sum_{m = -a}^{b-2}c_m z^m,
		\end{align*}
		where
	\[
	c_m = \#\{(j,k)\in[0,a]\times[1,b]: j-k = m\}.
	\]
	The sequence $(c_m)_m$ is weakly increasing between $m = -b$ and $m= - 1$, so $c_{-4} \leq c_{-2}$.  (These coefficients may be zero.)  This means
	\[
	\Re (c_{-4} z^{-4} + c_{-2}z^{-2}) \geq c_{-2}\Re z^{-2} - c_{-4}|\Re z^{-4}| \geq 0.
	\]
	It follows that
	\begin{align*}
		\Re \frac{(z)_b(z)_a}{z^a} &= \sum_{m = -a}^{b-2} c_m \Re z^m \\
		& \geq c_0 + c_{-1}\Re z^{-1} + \Re(c_{-2}z^{-2} + c_{-4}z^{-4}) \\
		&\geq \frac 43 > 1.15.
	\end{align*}
	Thus $\tfrac{z^a}{(z)_a(z)_b}\in \calE_d\subseteq E$.
	
	\item $\boldsymbol{b=1, a\geq 5{:}}$ Begin with $(*)$. Write
	\[
	\frac{z^a}{(z)_a} = \frac{z^a - 1 + 1}{(z)_a} = z - 1 + \frac{1}{(z)_a}.
	\]
	Because $z-1\in \ball{2}{1}^+$ and $|(z)_a| \geq 5$, we deduce 
	\[
	\frac{z^a}{(z)_a}\subseteq \ball{0}{1/5} + \ball{2}{1}^+ = \calE_c\subseteq E.
	\]
	
	\item $\boldsymbol{b=2, a\geq 5{:}}$ Begin with $(*)$.  Write
	\[
	\frac{z^a}{(z)_a(1+z)} = \frac{z-1}{z+1} + \frac{1}{(z)_a(z+1)}\in B_{1/10}\left(\frac{z-1}{z+1}\right).
	\]
	The function $\eta(z)\coloneqq \frac{z-1}{z+1}$ is a M{\"o}bius transformation with $\eta(1) = 0$ and $\eta(3) = \tfrac12$.  This means $\eta$ sends the ball $\ball{2}{1}$ to the ball $\ball{\tfrac14}{1/4}$.  It follows from Proposition \ref{prop:silli} that 
	\[
	B_{1/10}\left(\frac{z-1}{z+1}\right)\subseteq \ball{0.25}{ 0.35} = \calE_e\subseteq E.
	\]
	
	\item $\boldsymbol{b\in\{3,4\}, a\geq 5{:}}$ Begin with $(*)$.  Recall that $|(z)_a| \geq 5$ and $|(z)_b| \geq 3$ by Theorem \ref{thm:min-geom-series}.  Write
	\[
	\frac{z^a}{(z)_a(z)_b} = \frac{z-1}{(z)_b} + \frac{1}{(z)_a(z)_b}\in B_{1/15}\left(\frac{z-1}{(z)_b}\right).
	\]
	To bound the location of $\tfrac{z-1}{(z)_b}$, observe that its reciprocal is
	\begin{align}
		\frac{(z)_b}{z-1} &= \frac{z^{b-1} + \cdots + z + 1}{z-1} \nonumber \\ &= z^{b-2} + \cdots +(b-2)z + (b-1) + \frac{b}{z-1}. \label{eqn:recip-frac}
	\end{align}
	Because $b\leq 4$, Proposition \ref{prop:silli} implies that all terms in \eqref{eqn:recip-frac} have nonnegative real part, and furthermore $z\in \calB$ implies $\Re \tfrac{1}{z-1} \geq \tfrac12$.  It follows that 
	\[
	\Re\left(\frac{(z)_b}{z-1}\right) \geq (b-2) + (b-1) + \frac 12b = \frac 52b - 3 \geq \frac 92.
	\] 
	This means that $\tfrac{z-1}{(z)_b}\in \ball{\tfrac 19}{1/9}$, and so
	\[
	B_{1/15}\left(\frac{z-1}{(z)_b}\right) \subseteq B_{8/45}\left(\frac 19\right)\subseteq\calE_e\subseteq E.
	\]
	
	\item $\boldsymbol{b\in\{5,6\}{:}}$ Begin with $(*)$.  By Theorem \ref{thm:mitrinovic} and Lemma \ref{prop:ineq-smallcases},
	\[
	\left|\frac{z^a}{(z)_a(z)_b}\right|= \left|\frac{z^{a-1}}{(z)_a}\cdot \frac{z}{(z)_b}\right| \leq \left|\frac{z}{(z)_b}\right| \leq \frac 15.
	\]
	Thus $\tfrac{z^a}{(z)_a(z)_b}\in \calE_a\subseteq E$.
	
	\item $\boldsymbol{b\geq 7{:}}$ Begin with $(*^{-1})$.  We are finally able to use our more general inequalities.  Observe that 
	\[
	\abs{\frac{(z)_b}{z}} = \abs{\frac 1z + (z)_{b-1}} \geq \abs{(z)_{b-1}} - \frac{1}{|z|} \geq 4.5,
	\]
	so $|\tfrac{z}{(z)_b}| < 0.233$.  It follows that $|\tfrac{z^a}{(z)_a(z)_b}| \leq |\tfrac{z}{(z)_b}| < 0.233$ and $\tfrac{z^a}{(z)_a(z)_b}\in \calE_a\subseteq E$.
\end{itemize}

We have \textit{finally} covered all ordered pairs $(a,b)\in\N^2$, thus proving Theorem \ref{thm:criterion-nonzero}.
\end{proof}

\section{Comments and Future Work}
Because the sequence $\stern n\lm$ is relatively new, several unresolved questions regarding their properties and zero distributions persist.  These questions touch many areas of mathematics, including continued fraction theory and combinatorics.

\par While Theorem \ref{thm:parabola} was sufficient to show $\calB$ has no Stern zeros, it cannot be used to show that all Stern zeros have real part less than $1$.  There exist complex numbers $z$ with $\Re z \geq 1$ for which the set $\calA_z$ is not a subset of any parabolic region $E_\alpha$.  Figure \ref{fig:near-miss} illustrates $z = 1 + 2.5i$ as an example.  Any attempts to prove this stronger conjecture will need more general results from continued fraction theory.

\begin{figure}[ht]
\centering
\includegraphics[scale=0.67]{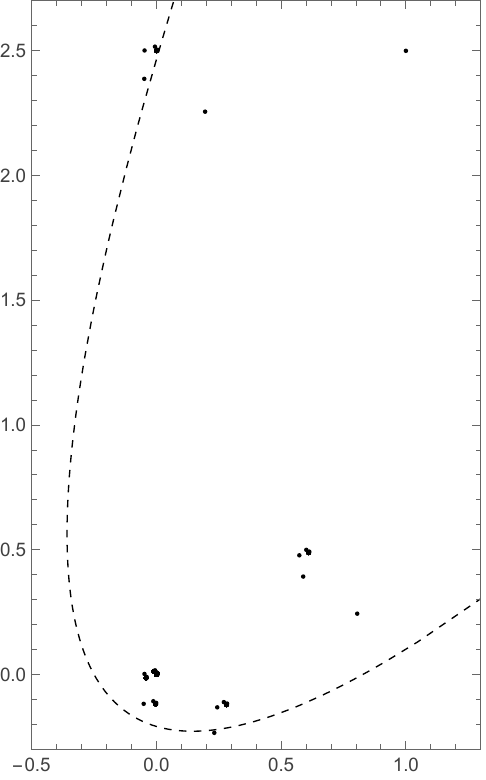}
\caption{\small{The set $\calA_{1+2.5i}$ (solid) compared with the near-miss region $E_{0.16\pi}$ (dashed).}}
\label{fig:near-miss}
\end{figure}

\par Stronger estimates for $|(z)_m|$ would also be desirable.  For example, it may be possible to strengthen Lemma \ref{lemma:ineq-push-left} to a statement of the form 
\[
|(x+yi)_n| \geq C|(x_0 + yi)_n|,
\] 
where $C = C_{n,y,x,x_0}$ is a positive number greater than $1$.  We hope that this stronger statement may lead to cleaner proofs and further generalizations of Theorem \ref{thm:min-geom-series}.

Finally, let us comment on a remark from the introduction of this paper.  In Section 1, we mentioned the stark contrast in the behavior of $\calS$ between the two half-planes $\{\Re w > 0\}$ and $\{\Re w < 0\}$.  Recent papers have aimed to make these differences explicit.  For example, while we conjecture that the set $\{\Re w: w\in \calS\}$ is bounded above, it is \textit{not} bounded below.

\begin{thm}[{\cite[Proposition 3.3]{Dilcher2017}}] The roots of $\stern{\alpha_n}\lm$, where $\alpha_n = \tfrac13(2^n - (-1)^n)$, are negative real numbers of the form 
\[
-\frac14\sec^2\left(\frac{j\pi}n\right), \quad j = 1,2,\ldots, \lfloor\tfrac{n-1}2\rfloor.
\]
Consequently, the interval $(-\infty,-\tfrac14]$ is dense in $\calS$.
\end{thm}

This leads into a related question: is the set of \textit{imaginary} parts bounded?  Surprisingly, the answer is also ``no''.

\begin{thm}
The polynomials $\stern{t_n}\lm$, where
\[
t_n = \frac{4^{n+1} + (-2)^n + 1}3,
\]
have roots $z_n\in\C$ satisfying $\Im z_n\to\infty$ as $n\to\infty$.
\end{thm}

The proof begins by finding a closed form for $\stern{t_n}\lm$ (involving the principal square root in $\mathbb C$), then uses Rouch{\'e}'s Theorem along with careful asymptotics to show there exist roots with arbitrarily large imaginary part.  The interested reader may consult the author's doctoral thesis \cite{Altizio2025} for the details.  The author believes $\Re z_n \asymp (\Im z_n)^2$, but has not worked out the details.

\bibliographystyle{abbrv}
\bibliography{stern-poly-bib}

\end{document}